\documentclass[11pt,reqno]{amsart}

\usepackage{amsmath}
\usepackage{amsthm}
\usepackage{amssymb}
\usepackage{amsfonts}
\usepackage{hyperref}
\usepackage[norsk,nameinlink]{cleveref}
\usepackage{xcolor}
\usepackage{tikz}
\usepackage{xcolor}
\usepackage{enumerate}
\usepackage{float}
\usepackage{stmaryrd}
\usepackage{mathrsfs}
\usepackage{caption}
\usepackage{enumitem}
\usetikzlibrary{intersections}
\usetikzlibrary{calc}

\newcommand{\quot}{/\kern-0.2em/}

\definecolor{DarkRed}{RGB}{200,20,20}
\definecolor{DarkGreen}{RGB}{0,120,0}
\definecolor{SkyBlue}{rgb}{0.16, 0.32, 0.75}
\hypersetup{
    colorlinks=true,
    linkcolor=DarkGreen, 
    linkbordercolor=DarkGreen,
    citecolor=DarkGreen,
    urlcolor=DarkGreen,
    }

\usepackage[left=3.5cm, right=3.5cm, top=1.4in, bottom=1.4in
]{geometry}

\newtheorem{theorem}{Theorem}[section]
\newtheorem{proposition}[theorem]{Proposition}
\newtheorem{lemma}[theorem]{Lemma}

\theoremstyle{definition}
\newtheorem{definition}[theorem]{Definition}

\newtheorem{remark}[theorem]{Remark}
\newtheorem*{acknowledgments}{Acknowledgments}

\makeatletter
\newcommand*{\rom}[1]{\expandafter\@slowromancap\romannumeral #1@}
\makeatother

\renewcommand{\epsilon}{\varepsilon}
\newcommand{\Mcr}{\overline{M}^{\mathrm{cr}}}
\newcommand{\Msp}{\overline{M}^{\mathrm{sp}}_{\bf b}}
\newcommand{\stackMsp}{\overline{\mathcal{M}}^{\mathrm{sp}}_{\bf b}}
\newcommand{\T}{\mathcal{T}}
\newcommand*{\sHom}{\mathcal{H}\kern -.5pt om}
\newcommand*{\sExt}{\mathcal{E}\kern -.5pt xt}

\crefname{section}{}{sections}
\crefname{lemma}{Lemma}{lemmas}
\crefname{enumi}{}{items}
\creflabelformat{enumi}{\textcolor{DarkGreen}{(#2#1#3)}}
\crefname{remark}{Remark}{remarks}
\crefname{theorem}{Theorem}{theorems}
\crefname{definition}{Definition}{definitions}
\crefname{proposition}{Proposition}{propositions}
\crefname{figure}{Figure}{figures}
\crefname{conjecture}{Conjecture}{conjectures}
\crefformat{equation}{\color{DarkGreen}(#2#1#3)}

\crefname{lemma}{Lemma}{Lemmas}
\crefname{proposition}{Proposition}{propositions}
\crefname{figure}{Figure}{figures}
\crefname{corollary}{Corollary}{corollaries}
\crefname{conjecture}{Conjecture}{conjectures}
\crefname{section}{Section}{Sections}
\crefname{subsection}{Section}{Sections}

\author{Hanlong Fang}
\address{H.~Fang, Assistant Professor, School of Mathematical Sciences, Peking University, 100871, Beijing, China}
\email{hanlongfang@math.pku.edu.cn}

\author{Luca Schaffler}
\address{L.~Schaffler, Assistant Professor, Dipartimento di Matematica e Fisica, Universit\`a degli Studi Roma Tre, Largo San Leonardo Murialdo 1, 00146, Roma, Italy}
\email{luca.schaffler@uniroma3.it}

\author{Xian Wu}
\address{X.~Wu, Visiting Scholar, School of Mathematical Sciences, Peking University, 100871, Beijing, China}
\email{xianwu.ag@gmail.com} 

\title{Fineness and smoothness of a KSBA moduli of marked cubic surfaces}


\begin{document}

\begin{abstract}
By work of Gallardo--Kerr--Schaffler, it is known that Naruki’s compactification of the moduli space of marked cubic surfaces is isomorphic to the normalization of the Koll\'ar, Shepherd-Barron, and Alexeev compactification parametrizing pairs $\left(S,\left(\frac{1}{9}+\epsilon\right)D\right)$, with $D$ the sum of the $27$ marked lines on $S$, and their stable degenerations. In the current paper, we show that the normalization assumption is not necessary as we prove that this KSBA compactification is smooth. Additionally, we show it is a fine moduli space. This is done by studying the automorphisms and the $\mathbb{Q}$-Gorenstein obstructions of the stable pairs parametrized by it.
\end{abstract}

\maketitle


\section{Introduction}

In the 19th century, Cayley and Salmon discovered that a smooth complex cubic surface in $\mathbb{P}^3$ contains exactly $27$ lines \cite{cayley1849triple}. This classical result in algebraic geometry is at the center of many of the geometric properties of a cubic surface. For instance, Naruki showed that the isomorphism class of a cubic surface can be reconstructed from the cross-ratios of the quadruplets of collinear tritangents (planes that intersect the cubic surface in three lines), and used this information to construct a smooth and simple normal crossing compactification, here denoted by $\Mcr$, of the moduli space $M$ of marked cubic surfaces \cite{naruki1982cross}.

More recently, the study of compactifications of the moduli space of cubic surfaces has drawn a lot of attention (see for instance \cite{allcock2002complex,casalaina2024non}). In particular, Naruki's compactification has been revisited and reinterpreted in view of new advances within moduli theory. In this sense, our starting point is the compactification theory for moduli spaces of higher dimensional algebraic varieties originated by work of Koll\'ar, Shepherd-Barron, and Alexeev, which is nowadays referred to as the \emph{KSBA compactification} \cite{kollar1988threefolds,Alexeev+1996+1+22}. In particular, one can consider the KSBA compactification $M\subseteq\overline{M}^{\mathrm{snc}}$ parametrizing stable pairs $(S,D)$, where $S$ is a smooth cubic surface and $D$ is the sum of the $27$ marked lines with simple normal crossings, and their stable degenerations. Then, Hacking--Keel--Tevelev proved that $\Mcr$ is isomorphic to the log canonical model of $\overline{M}^{\mathrm{snc}}$ \cite[Theorem~1.5]{hacking2009stable}.

Additionally, it was conjectured in \cite[Remark~1.3~(4)]{hacking2009stable} that $\Mcr$ admits a modular interpretation in terms of some other stable pairs $(S,\mathbf{b}D)$, where $\mathbf{b}D:=\sum_{i=1}^{27}\left(\frac{1}{9}+\epsilon\right)D_i$, $\epsilon\in\mathbb{Q}$, $0<\epsilon\ll 1$. This was confirmed by Gallardo--Kerr--Schaffler \cite[Theorem~1.5]{gallardo2021geometric}. More precisely, they showed that $\Mcr$ is isomorphic to the normalization of the KSBA compactification $\Msp$ of the moduli space of pairs $(S,\mathbf{b}D)$. The fundamental observation is that the family $(\overline{\mathcal{S}},\overline{\mathcal{D}})\rightarrow\Mcr$ constructed by Naruki and Sekiguchi \cite{naruki1980modification,naruki1982cross} is actually a family of stable pairs for the weight $\frac{1}{9}+\epsilon$ (see \cite[Proposition~5.12]{gallardo2021geometric}). Later on, the study of the wall crossings between $\Msp$ and $\overline{M}^{\mathrm{snc}}$ was carried out by Schock in \cite{schock2023moduli}.

The goal of the current paper is twofold. By computing the automorphisms of the stable pairs parametrized by $\Msp$ and their $\mathbb{Q}$-Gorenstein obstructions, we prove the following result.

\begin{theorem}[{=\cref{thm:fiberwise_argument}+\cref{thm:smoothstack}}]
\label{Thm:main}
For all stable pairs $(S,\mathbf{b}D)$ parametrized by the KSBA compactification $\Msp$ we have that:
\begin{enumerate}[label=$\mathrm{(\arabic*)}$]
    \item $(S,\mathbf{b}D)$ has trivial automorphism group;
    \item The $\mathbb{Q}$-Gorenstein obstruction space of $(S,\mathbf{b}D)$ vanishes. In particular,
    \[
    H^2(\mathcal{T}_S^0)=H^1(\mathcal{T}_S^1)=H^0(\mathcal{T}_S^2)=0,
    \]
    where $\mathcal{T}_S^0,\mathcal{T}_S^1,\mathcal{T}_S^2$ are respectively the tangent, deformation, and obstruction sheaves (see \cref{def:ordinarydefobs}).
\end{enumerate}
Therefore, combining the two results above, the KSBA compactification $\Msp$ is a fine and smooth moduli space.
\end{theorem}

In particular, Naruki's compactification $\Mcr$ is isomorphic to the KSBA compactification $\Msp$, without need of normalizing the latter. In complete generality, the KSBA compactification of a moduli space is not expected to be smooth, and it may have arbitrary singularity types \cite{mnev1985varieties,vakil2006murphy}. For instance, this is the case for the moduli space of hyperplane arrangements \cite{keel2006geometry,hacking2006compactification,alexeev2015moduli}. So, the smoothness of the KSBA compactification $\Msp$ stands out in this sense. To prove this, we work with the corresponding moduli stack $\stackMsp$ and we show that it is smooth using Hacking's ideas in \cite{hacking2004compact}, which originated in the work of Illusie \cite{illusie1971complexe,illusie1972complexe,Illusie1972cotangent}. The claimed smoothness is implied by appropriate cohomological calculations that involve the cotangent complex of the degenerations being parametrized: see \cref{sec:smooth-stack}. This strategy was also used by Deopurkar--Han in \cite{deopurkar2021stable} for non-normal degenerations with double normal crossings. Independently, the local smoothness of a moduli stack of stable surfaces at some degenerations with pinch points was proved in \cite{fantechi2022smoothing} by Fantechi--Franciosi--Pardini.

Likewise, in general, the KSBA compactification of a moduli space is not expected to be a fine moduli space. This is due to the presence of nontrivial automorphisms of the objects parametrized by the compactification. However, the marking of the $27$ lines prevents such automorphisms to appear: we discuss this is the proof of \cref{prop:auto}.

We work over the field of complex numbers.

\begin{acknowledgments}
We would like to thank Patricio Gallardo, Edoardo Sernesi, and Filippo Viviani for helpful conversations. H. Fang and X. Wu are supported by National Key R\&D Program of China under Grant No.2022YFA1006700 and NSFC-12201012. X.Wu previously was also supported by OPUS grant National Science Centre, Poland grant UMO-2018/29/BST1/01290. L. Schaffler was supported by the projects ``Programma per Giovani Ricercatori Rita Levi Montalcini'', PRIN2020KKWT53 ``Curves, Ricci flat Varieties and their Interactions'', and PRIN 2022 ``Moduli Spaces and Birational Geometry'' – CUP E53D23005790006. L. Schaffler is a member of the INdAM group GNSAGA.
\end{acknowledgments}


\section{Fineness of the moduli space}
\label{sec:fine-moduli-space}

\subsection{Preliminary definitions}

We begin by briefly recalling from \cite[Section~5.2]{gallardo2021geometric} the definitions of the main objects of interest in the current paper: the moduli stack $\stackMsp$ and its coarse moduli space $\Msp$. For the background on stable pairs in the sense of the minimal model program and their moduli, we refer to \cite{kollar2013singularities,kollar2023families}. For the reader's convenience, first we recall the definition of Viehweg's moduli stack adopted in \cite[Definition~5.3]{gallardo2021geometric}.

\begin{definition}
Let us fix $d,N\in\mathbb{Z}_{>0}$, $C\in\mathbb{Q}_{>0}$, and a weight vector $\mathbf{b}=(b_1,\ldots,b_n)$ with $b_i\in(0,1]\cap\mathbb{Q}$ and $Nb_i\in\mathbb{Z}$ for all $i\in\{1,\ldots,n\}$. For a reduced complex scheme $S$, let $\overline{\mathcal{V}}(S)$ be the set of proper flat families $(\mathcal{X},\mathcal{D}=\sum_ib_i\mathcal{D}_i)\rightarrow S$ such that:
\begin{enumerate}

\item For all $i\in\{1,\ldots,n\}$, $\mathcal{D}_i$ is a closed subscheme of codimension one such that $\mathcal{D}_i\rightarrow S$ is flat at the generic points of $\mathcal{X}_s\cap\mathrm{Supp}(\mathcal{D}_i)$ for every $s\in S$;

\item Every geometric fiber $(X,D)$ is a stable pair, i.e., $(S,\mathbf{b}D)$ is semi-log canonical \cite[Definition–Lemma~5.10]{kollar2013singularities} and the $\mathbb{Q}$-Cartier divisor $K_S+\mathbf{b}D$ is ample. Moreover, $(X,D)$ satisfies $\dim(X)=d$ and $(K_X+D)^d=C$;

\item There is an invertible sheaf $\mathcal{L}$ on the family $\mathcal{X}$ such that, for every geometric fiber $(X,D)$, we have that $\mathcal{L}|_X\cong\mathcal{O}_X(N(K_X+D))$.

\end{enumerate}
A stack $\overline{\mathcal{V}}$ as above is called \emph{Viehweg's moduli stack}. We have that $\overline{\mathcal{V}}$ is a separated Deligne--Mumford stack which admits a projective coarse moduli space.
\end{definition}

\begin{definition}
\label{def:moduli-stack}
Let $\overline{\mathcal{V}}$ be the Viehweg's moduli stack for $d=2$, $\mathbf{b}=(\frac{1}{9}+\epsilon,\ldots,\frac{1}{9}+\epsilon)\in\mathbb{Q}^{27}$ with $0<\epsilon\ll1$, and $C=243\epsilon^2$ (see \cite[Definition~5.3]{gallardo2021geometric}). Let us consider the pairs $(S,D)$, where $S\subseteq\mathbb{P}^3$ is a smooth cubic surface whose $27$ lines have simple normal crossings and $D=\sum_{i=1}^{27}D_i$ is the sum of these marked lines. Then, the pairs
\[
\left(S,\mathbf{b}D:=\sum_{i=1}^{27}\left(\frac{1}{9}+\epsilon\right)D_i\right)
\]
are stable. Moreover, $(K_S+\mathbf{b}D)^2=243\epsilon^2$ and the pairs $(S,\mathbf{b}D)$ vary in a family $(\mathcal{S},\mathbf{b}\mathcal{D})\rightarrow M$ over a quasi-projective moduli space with pairwise non-isomorphic fibers (see \cite[Remark~5.5]{gallardo2021geometric}). As the pairs $(S,\mathbf{b}D)$ are parametrized by the Viehweg's moduli stack $\overline{\mathcal{V}}$, this family induces a morphism $M\rightarrow\overline{\mathcal{V}}$, and we denote by $\stackMsp$ the closure of the image of $M$ inside $\overline{\mathcal{V}}$. We let $\Msp$ be the coarse moduli space associated to $\stackMsp$.
\end{definition}

\begin{remark}
By the work in \cite[Theorem~1.1]{schock2023moduli}, we have that the KSBA compactification $\Msp$ stays the same and the objects parametrized by it remain stable if we change the weight vector $\mathbf{b}=\left(\frac{1}{9}+\epsilon,\ldots,\frac{1}{9}+\epsilon\right)\in\mathbb{Q}^{27}$ to $(c,\ldots,c)\in\mathbb{Q}^{27}$ for all $c\in\left(\frac{1}{9},\frac{1}{6}\right]$.
\end{remark}

\begin{remark}
To ease the correspondence between different notations being used, in \cite{gallardo2021geometric}, the moduli spaces $\Mcr$, $\Msp$, and the stack $\stackMsp$ are denoted by $\overline{\mathbf{N}}$, $\overline{\mathbf{Y}}_{\frac{1}{9}+\epsilon}$, and $\overline{\mathcal{Y}}_{\frac{1}{9}+\epsilon}$, respectively.
\end{remark}

\subsection{The stable pairs parametrized by \texorpdfstring{$\stackMsp$}{Lg}}\label{Subsection:stable_pairs_para}

We prove that the KSBA compactification $\Msp$ is a fine moduli space by showing that every object parametrized by it has the identity as its only automorphism.

The stable pairs parametrized by $\Msp$ are precisely the fibers of the family $(\overline{S},\mathbf{b}\overline{\mathcal{D}})$ of marked cubic surfaces over Naruki's compactification $\Mcr$ in \cite[Sections~5.3, 5.4]{gallardo2021geometric}. The line arrangement $D$ on $S$ are partially pictured in \cref{Fig:line_arrangement} for all possible degeneration types of $S$. Following \cite[Theorem~1.5]{gallardo2021geometric}, the degeneration type of a pair $(S,\mathbf{b}D)$ is defined via the boundary divisors of $\Mcr$ as follows:
\begin{enumerate}

\item There are $36$ divisors associated to the $A_1$-subroot systems of $E_6$. The intersection of $k$ such divisors parametrizes pairs $(S,\mathbf{b}D)$ such that $S$ is a normal cubic surfaces with $k$ $A_1$-singularities. We call such pairs of \emph{type $A_1^k$}.

\item There are $40$ divisors associated to the $A_2^3$-subroot systems of $E_6$, which are pairwise disjoint. Generically, they parametrize stable pairs of \emph{type $N$}, which are given by $(S,\mathbf{b}D)$, where $S=S_0\cup S_1\cup S_2\subseteq\mathbb{P}^3$, $S_i=V(x_i)$ for $i=0,1,2$, and the $27$ lines are distinct. We say that $(S,\mathbf{b}D)$ is of \emph{type $(A_1^k,N)$} if it is parametrized by the intersection of one type $A_2^3$ divisor with $k$ type $A_1$ divisors.

\end{enumerate}
Solid, dashed, and dotted lines in \cref{Fig:line_arrangement} are explained as follows:
\begin{enumerate}

\item A thin solid line denotes a reduced line on $S$.

\item A dashed line denotes a line on $S$ with multiplicity $2$.

\item A dotted line denotes a line on $S$ with multiplicity $4$.

\item A black dot denotes an $A_1$-singularity.

\item In the second row, there are three thick lines in each of the four pictures, which denote the coordinate lines $\Delta_0=V(x_1,x_2),\Delta_1=V(x_0,x_2)$, and $\Delta_2=V(x_0,x_1)$. The areas between two thick lines denote appropriate coordinate planes.

\end{enumerate}

\begin{figure}[h]
    \centering
    \begin{tabular}{cccc}

        \begin{minipage}{0.23\textwidth}
            \centering
            \vspace{5pt}
            \begin{tikzpicture}[scale=0.5]
                \fill (0,0) circle (5pt);
                \foreach \a in {10, 42, 74, 106, 138, 170} {
                    \draw[dashed, line width=0.8pt] (0,0) -- (\a:3);
                }
            \end{tikzpicture}
        \end{minipage} &
        
        \begin{minipage}{0.23\textwidth}
            \centering
            \vspace{5pt}
            \begin{tikzpicture}[scale=0.5]
                \draw[dotted, line width=0.8pt] (-3,0) -- (3,0);
                \foreach \x in {-2,2} {
                    \fill (\x,0) circle (5pt);
                    \foreach \a in {60,80,100,120} {
                        \draw[dashed, line width=0.8pt] (\x,0) -- +(\a:3);
                    }
                }
            \end{tikzpicture}
        \end{minipage} &
        
        \begin{minipage}{0.23\textwidth}
            \centering
            \vspace{5pt}
            \begin{tikzpicture}[scale=0.5]
                \pgfmathsetmacro{\side}{6/sqrt(3)}
                \coordinate (A) at (-\side/2,0);
                \coordinate (B) at (\side/2,0);
                \coordinate (C) at (0,3);
                \draw[dotted, line width=0.8pt] (A) -- (B) -- (C) -- cycle;
                \foreach \v in {A, B, C} {
                    \fill (\v) circle (5pt);
                }
                \draw[dashed, line width=0.8pt] (A) -- +(20:1.5) (A) -- +(40:1.5);
                \draw[dashed, line width=0.8pt] (B) -- +(140:1.5) (B) -- +(160:1.5);
                \draw[dashed, line width=0.8pt] (C) -- +(-80:1.5) (C) -- +(-100:1.5);
            \end{tikzpicture}
        \end{minipage} &
        
        \begin{minipage}{0.23\textwidth}
            \centering
            \vspace{5pt}
            \begin{tikzpicture}[scale=0.5]
                \foreach \y in {0,2} {
                    \draw[dotted, line width=0.8pt] (-3,\y) -- (3,\y);
                    \fill (-2,\y) circle (5pt) (2,\y) circle (5pt);
                }
                \draw[dotted, line width=0.8pt] (-3,-0.5) -- (3,2.5) (-3,2.5) -- (3,-0.5);
                \draw[dotted, line width=0.8pt] (-2,-1) -- (-2,3) (2,-1) -- (2,3);
            \end{tikzpicture}
        \end{minipage} \\
        \vspace{5pt}
        $A_1$ & $A_1^2$ & $A_1^3$ & $A_1^4$\\

        \begin{minipage}{0.23\textwidth}
            \centering
            \vspace{5pt}
            \begin{tikzpicture}[scale=0.5]
                \draw[line width=1.5pt] (0,0) -- (0,4);                
                \draw[line width=1.5pt] (0,0) -- (-30:4);                
                \draw[line width=1.5pt] (0,0) -- (210:4);
                \foreach \i in {1,2,3} {
                    \foreach \j in {1,2,3} {
                        \draw (0,\i) -- (-30:\j) (0,\i) -- (210:\j) (-30:\i) -- (210:\j);
                    }
                }
            \end{tikzpicture}
        \end{minipage} &
        
        \begin{minipage}{0.23\textwidth}
            \centering
            \vspace{5pt}
            \begin{tikzpicture}[scale=0.5]
                \draw[line width=1.5pt] (0,0) -- (0,4);                
        \draw[line width=1.5pt] (0,0) -- (-30:4);                
        \draw[line width=1.5pt] (0,0) -- (210:4);
        \coordinate (P1) at (0,1);
        \coordinate (P2) at (0,2);
        \coordinate (P3) at (0,3);

        \coordinate (Q1) at (-30:1);
        \coordinate (Q2) at (-30:2);
        \coordinate (Q3) at (-30:3);

        \coordinate (R1) at (210:1);
        \coordinate (R2) at (210:2);
        \coordinate (R3) at (210:3);

        \foreach \p in {P1, P2, P3} {
            \foreach \q in {Q1, Q2, Q3} {
                \draw (\p) -- (\q);
            }
            \foreach \r in {R1} {
                \draw (\p) -- (\r);
            }
        }
        \foreach \q in {Q1, Q2, Q3} {
            \foreach \r in {R1} {
            \draw (\q) -- (\r);
            }
        }
        \foreach \p in {P1, P2, P3} {
            \foreach \r in {R3} {
                \draw [dashed] (\p) -- (\r);
            }
        }
        \foreach \q in {Q1, Q2, Q3} {
            \foreach \r in {R3} {
                \draw [dashed] (\q) -- (\r);
            }
        }
            \end{tikzpicture}
        \end{minipage} &
        
        \begin{minipage}{0.23\textwidth}
            \centering
            \vspace{5pt}
            \begin{tikzpicture}[scale=0.5]
                \draw[line width=1.5pt] (0,0) -- (0,4);                
        \draw[line width=1.5pt] (0,0) -- (-30:4);                
        \draw[line width=1.5pt] (0,0) -- (210:4);

        \coordinate (P1) at (0,1);
        \coordinate (P2) at (0,2);
        \coordinate (P3) at (0,3);

        \coordinate (Q1) at (-30:1);
        \coordinate (Q2) at (-30:2);
        \coordinate (Q3) at (-30:3);

        \coordinate (R1) at (210:1);
        \coordinate (R2) at (210:2);
        \coordinate (R3) at (210:3);

        \foreach \p in {P1, P2, P3} {
            \foreach \q in {Q1} {
                \draw (\p) -- (\q);
            }
            \foreach \r in {R1} {
                \draw (\p) -- (\r);
            }
        }
        \foreach \p in {P1, P2, P3} {
            \foreach \r in {R3} {
                \draw [dashed] (\p) -- (\r);
            }
        }
        \foreach \p in {P1, P2, P3} {
            \foreach \r in {Q3} {
                \draw [dashed] (\p) -- (\r);
            }
        }
        \draw [dashed] (Q1) -- (R3);
        \draw [dashed] (Q3) -- (R1);
        \draw [dotted, line width=0.8pt] (Q3) -- (R3);
        \draw (210:1) -- (-30:1);
        \end{tikzpicture}
        \end{minipage} &
        
        \begin{minipage}{0.23\textwidth}
            \centering
            \vspace{5pt}
            \begin{tikzpicture}[scale=0.5]
                \draw[line width=1.5pt] (0,0) -- (0,4);                
                \draw[line width=1.5pt] (0,0) -- (-30:4);                
                \draw[line width=1.5pt] (0,0) -- (210:4);

                \draw (0,1) -- (210:1) -- (-30:1) -- cycle;
                \draw[dashed] (0,1) -- (210:3) -- (-30:1) -- (0,3) -- (210:1) -- (-30:3) -- cycle;
                \draw[dotted, line width=0.8pt] (0,3) -- (-30:3) -- (210:3) -- cycle;
            \end{tikzpicture}
        \end{minipage} \\
        \vspace{5pt}
        $N$ & $(A_1,N)$ & $(A_1^2,N)$ & $(A_1^3,N)$\\

    \end{tabular}
    \caption{The limits of the $27$ lines on the degenerations parametrized by $\Msp$. The pictures are extracted from \cite[Table~1]{gallardo2021geometric}.}
    \label{Fig:line_arrangement}
\end{figure}

\subsection{Proof of fineness}

\begin{definition}
Let $(S,\mathbf{b}D)$ be a stable pair parametrized by $\Msp$. We denote by $\mathrm{Aut}(S,D)$ the automorphism group of the pair $(S,\mathbf{b}D)$. That is, $\mathrm{Aut}(S,D)$ is the subgroup of automorphism $f\in\mathrm{Aut}(S)$ such that $f(D_i)=D_i$ for all $i=1,\ldots,27$.
\end{definition}

\begin{proposition}
\label{prop:auto}
Let $(S,\mathbf{b}D)$ be a stable pair parametrized by $\Msp$. Then, we have
\[
\mathrm{Aut}(S,D)=\{\mathrm{id}_S\}.
\]
\end{proposition}

\begin{proof}
Let $f\in\mathrm{Aut}(S,D)$. Let us first discuss the case where $S$ is smooth or $(S,\mathbf{b}D)$ is of type $A_1^k$, $k=1,\ldots,4$. In these cases, we have that $f$ is the identity on $S$ if and only if it fixes point-wise five points in $S$ which are linearly general in $\mathbb{P}^3$. This is true because $f$ extends to a linear map $\widetilde{f}\in\mathrm{PGL}_4(\mathbb{C})$ by \cite[Theorem~30]{kollar2019algebraic}, and $\widetilde{f}$ is the identity if and only if it fixes point-wise five linearly general points. As $f$ sends a marked line to itself, the points arising as the intersection of two marked lines will be fixed by $f$. So, it suffices to find five linearly general points on $S$ arising as the intersection of two lines. We do this case by case.

\begin{enumerate}[leftmargin=*]

\item \textbf{Smooth type.} The surface $S$ can be viewed as the blow up of $\mathbb{P}^2$ at six points $p_1,\ldots,p_6$ such that no three points are collinear and all six do not lie on a conic. Let $E_i\subseteq S$ be the exceptional divisor over $p_i$. For $i\neq j$, let $\widehat{\ell}_{ij}$ be the strict transform of the line $\ell_{ij}\subseteq\mathbb{P}^2$ passing through $p_i$ and $p_j$. Finally, let $\widehat{C}_i$ be the strict transform of the unique conic $C_i\subseteq\mathbb{P}^2$ passing through $p_1,\ldots,\widehat{p}_i,\ldots,p_6$. Denote by $P_{ij}$ the intersection point $E_i\cap\widehat{\ell}_{ij}$ and by $Q_{ij}$ the intersection point $E_i\cap\widehat{C}_j$. Then, the four points $P_{12},Q_{13}\in E_1$ and $P_{21},Q_{23}\in E_2$ are not coplanar: if by contradiction they were contained in a plane $H\subseteq\mathbb{P}^3$, then the plane $H$ would contain the four lines $E_1,E_2,\widehat{\ell}_{12},\widehat{C}_3$, which is impossible because $S\cap H$ contains at most three lines.

So, the triples of distinct points among $P_{12},Q_{13},P_{21},Q_{23}$ generate four distinct planes, which we name $F_1,\ldots,F_4$. These planes intersect the line $E_3$ in four points. Hence, there exists $P\in\{P_{31},P_{32},P_{34},P_{35},P_{36}\}$ such that $P$ is different from $E_3\cap F_i$ for all $i=1,\ldots,4$. It follows from this construction that $P_{12},Q_{13},P_{21},Q_{23},P$ are five linearly general points arising as the intersection points of pairs of lines.

\item \textbf{Type $A_1$.} The cubic surface $S$ can be constructed as the blow up of $\mathbb{P}^2$ at six distinct points $p_1,\ldots,p_6$ lying on a conic curve $C$, then contract the strict transform of $C$ to an $A_1$-singularity, which we denote by $P$. By \cite[Page~383]{Nguyen2005semi}, the six double lines through $P$ are the exceptional curves $E_1,\ldots,E_6$ (we use the same notation for their images under the contraction). We can assume that the plane spanned by $E_1,E_2$ is different from the one spanned by $E_3,E_4$: otherwise, if $E_1,\ldots,E_4$ were contained in a plane $H$, then the intersection $S\cap H$ would contain four lines, which cannot be.

On $S$, the strict transform of the line connecting $p_i$ and $p_j$, denoted by $\widehat{\ell}_{ij}$, intersects $E_i$ at a point $P_{ij}$. Consider the four points $P_{12},P_{21},P_{34},P_{43}$ (this is pictured in \cref{Fig:A1_local}). We have two possibilities:

\begin{itemize}

\item $P_{12},P_{21},P_{34},P_{43}$ are not coplanar. Then, the quintuple $P_{12},P_{21},P_{34},P_{43},P$ is in general linear position. To show this, it is enough to check that there is no plane containing $P$ and any three among $P_{12},P_{21},P_{34},P_{43}$. If there was a plane $H$ containing $P$ and three among $P_{12},P_{21},P_{34},P_{43}$, then $H\cap S$ would contain three distinct exceptional divisors $E_i,E_j,E_k$, $i,j,k\in\{1,\ldots,4\}$. But then, $H$ would also contain the lines $\widehat{\ell}_{ij},\widehat{\ell}_{ik},\widehat{\ell}_{jk}$, which cannot be.

\item $P_{12},P_{21},P_{34},P_{43}$ are contained in a plane $T$. Then, consider the line $\widehat{\ell}_{23}\subseteq S$ and the point $P_{23}=E_2\cap\widehat{\ell}_{23}$. Then, the four points $P_{12},P_{23},P_{34},P_{43}$ are not coplanar: if they were contained in a plane, then this plane would equal $T$, as $T$ is spanned by $P_{12},P_{34},P_{43}$. Then, the plane $T$ would contain the line $E_2$ because $P_{21},P_{23}\in T\cap E_2$. Hence, also $P\in T$, so that $T\cap S$ would contain (at least) the lines $E_1,E_2,E_3,E_4$, which is impossible.

We now claim that the five points $P_{12},P_{23},P_{34},P_{43},P$ are in general linear position, and this can be argued in a way analogous to what we did at the end of the previous bullet point.

\end{itemize}

\begin{center}
\begin{figure}[h]
\begin{tikzpicture}[scale=0.7]
    \draw (0,-0.2) -- (0,3) -- (4,4) -- (4,0.8) -- cycle;
    \draw (0,3) -- (-4,4) -- (-4,0.8) -- (0,-0.2);
    \fill (0,1) circle (5pt);
    \coordinate (P) at (0,1);
    \coordinate (E1) at (-3,2);
    \coordinate (E2) at (-3,3);
    \coordinate (E3) at (3,3);
    \coordinate (E4) at (3,2);
    \coordinate (A) at (-2.5,1);
    \coordinate (B) at (-2.5, 3.5);
    \coordinate (C) at (2.5,1);
    \coordinate (D) at (2.5, 3.5);
    \coordinate (E) at (-2,2);
    \coordinate (F) at (2,2);
    \draw [dashed, line width=1pt] (P) -- (E1);
    \draw [dashed, line width=1pt] (P) -- (E2);
    \draw [dashed, line width=1pt] (P) -- (E3);
    \draw [dashed, line width=1pt] (P) -- (E4);
    \node [left] at (-3,2) {$E_1$};
    \node [left] at (-3,3) {$E_2$};
    \node [right] at (3,3) {$E_3$};
    \node [right] at (3,2) {$E_4$};
    \draw [dashed, line width=1pt] (A) -- (B);
    \draw [dashed, line width=1pt] (C) -- (D);
    \node [left] at (-2.4,1.1) {$\widehat{\ell}_{12}$};
    \node [right] at (2.4,1.1) {$\widehat{\ell}_{34}$};
    \node [right] at (0,0.5) {$P$};
    \draw [dashed, line width=1pt] (E) -- (F);
    \node [above] at (-1,2) {$\widehat{\ell}_{23}$};
    \path[name path=line1] (P) -- (E1);
    \path[name path=line2] (P) -- (E2);
    \path[name path=line3] (P) -- (E3);
    \path[name path=line4] (P) -- (E4);
    \path[name path=lineleft] (A) -- (B);
    \path[name path=lineright] (C) -- (D);
    \path[name path=linehorizontal] (E) -- (F);
    \path[name intersections={of=line1 and lineleft, by=pp1}];
    \path[name intersections={of=line2 and lineleft, by=pp2}];
    \path[name intersections={of=line3 and lineright, by=pp3}];
    \path[name intersections={of=line4 and lineright, by=pp4}];
    \path[name intersections={of=line2 and linehorizontal, by=pp5}];
    \path[name intersections={of=line3 and linehorizontal, by=pp6}];
    \fill (pp1) circle (3pt);
    \fill (pp2) circle (3pt);
    \fill (pp3) circle (3pt);
    \fill (pp4) circle (3pt);
    \fill (pp5) circle (3pt);
    \fill (pp6) circle (3pt);
\end{tikzpicture}
\caption{Some of the lines on a cubic surface with exactly one $A_1$ singularity in case (2) of the proof of \cref{prop:auto}.}
\label{Fig:A1_local}
\end{figure}
\end{center}

\item \textbf{Type $A_1^2$.} Denote the two $A_1$-singularities by $P_1,P_2$ and the four double lines through $P_i$ by $L_i(j)$, $j=1,\ldots,4$, $i=1,2$. According to Cayley's table in \cite[Page~269]{cayley1869vii} (this is our case because the multiplicities of the lines correspond), every double line passing through $P_1$ intersects with exactly one double line through $P_2$ (see the left-hand side of \cref{Fig:A1^2A1^3line_arrangement}). So, up to relabeling, we may assume that $L_1(j)$ intersects with $L_2(j)$ at a point $Q_j$ for $j=1,\ldots,4$. Notice that the points $Q_1,\ldots,Q_4$ cannot lie on the same line. Otherwise, there would be a plane $H$ containing the lines $L_1(1),\ldots,L_1(4)$, hence the intersection $H\cap S$ would consist of at least four lines, which is impossible. So, we can select three points among $Q_1,\ldots,Q_4$ which do not lie on a line. Without loss of generality, suppose that these three points are $Q_1,Q_2,Q_3$. Then, $Q_1,Q_2,Q_3,P_1,P_2$ are in general linear position. To show this, as $Q_1,Q_2,Q_3$ are not aligned, it is enough to show, up to relabeling, that $Q_1,Q_2,P_1,P_2$ are not coplanar. By contradiction, if $Q_1,Q_2,P_1,P_2$ were contained in a plane $H$, then $H\cap S$ would contain the lines $L_1(1),L_1(2),L_2(1),L_2(2)$, which is impossible.

\item \textbf{Type $A_1^3$.} The configuration of lines on $S$ is described in \cite[Page~290]{cayley1869vii} (this is our case because the multiplicities of the lines correspond). In this table, the double lines are labeled by $1,\ldots,6$, so here we denote them by $L_1,\ldots,L_6$. From the same table, we see that the three $A_1$ singularities are $P_1=L_1\cap L_2$, $P_2=L_3\cap L_4$, and $P_3=L_5\cap L_6$. Moreover, the lines $L_1$ and $L_3$ are coplanar, so that we have the intersection point $Q_1=L_1\cap L_3$, distinct from $P_1,P_2,P_3$. Analogously, we have the intersection point $Q_2=L_2\cap L_4$. For a visual reference, see the right-hand side of \cref{Fig:A1^2A1^3line_arrangement}. Then, the five points $P_1,P_2,P_3,Q_1,Q_2$ are in general linear position. To prove this, let $\mathcal{X}\subseteq\{P_1,P_2,P_3,Q_1,Q_2\}$ be any subset of four points and assume by contradiction that $\mathcal{X}$ is contained in a plane $H$. We now argue by cases.

\begin{itemize}

\item $\mathcal{X}=\{P_1,P_2,P_3,Q_i\}$. Then, $H\cap S$ would have to contain the three quadruple lines and $L_1$ if $i=1$, or $L_4$ if $i=2$. So, this is excluded.

\item $\mathcal{X}=\{P_1,P_2,Q_1,Q_2\}$. This is also impossible because $H\cap S$ would have to contain the lines $L_1,\ldots,L_4$.

\item $\mathcal{X}=\{P_1,P_3,Q_1,Q_2\}$. Then, $H\cap S$ would have to contain $L_1,L_2$, the quadruple line through $P_1$ and $P_3$, and the line $L_6$ because the lines $L_1$ and $L_6$ have non-empty intersection (this comes from \cite[Page~290]{cayley1869vii}). This is impossible.

\item $\mathcal{X}=\{P_2,P_3,Q_1,Q_2\}$. Then, $H\cap S$ would have to contain $L_3,L_4$, the quadruple line through $P_2$ and $P_3$, and the line $L_5$ because the lines $L_3$ and $L_5$ have non-empty intersection (this can be read from \cite[Page~290]{cayley1869vii}). This is also excluded.

\end{itemize}
As no quadruple in $P_1,P_2,P_3,Q_1,Q_2$ lies on a plane, the quintuple is linearly general.

\begin{figure}[h]
    \centering
\begin{tabular}{cc}
\begin{minipage}{0.4\textwidth}
\begin{tikzpicture}[scale=0.85]
    \draw[dotted, line width=0.8pt] (-3,0) -- (3,0);
    \draw[dashed] (-2,0) -- (0,0.6) -- (2,0);
    \draw[dashed] (-2,0) -- (0,1.2) -- (2,0);
    \draw[dashed] (-2,0) -- (0,2) -- (2,0);
    \draw[dashed] (-2,0) -- (0,3) -- (2,0);
    \fill (0,0.6) circle (1.5pt);
    \fill (0,1.2) circle (1.5pt);
    \fill (0,2) circle (1.5pt);
    \fill (0,3) circle (1.5pt);
    \fill (2,0) circle (3.7pt);
    \fill (-2,0) circle (3.7pt);
\end{tikzpicture}
\end{minipage}
&
\begin{minipage}{0.2\textwidth}
\begin{tikzpicture}[scale=1.0]

    \coordinate (A) at (0,0);
    \coordinate (B) at (3.4,0);
    \coordinate (C) at (1.7,{1.7*sqrt(3)});

    \draw[dotted, thick] (A) -- (B) -- (C) -- cycle;

    \coordinate (MA) at ($(B)!0.5!(C)$);
    \coordinate (PA) at ($(MA) + (-1.6,-1)$);
    \coordinate (MB) at ($(A)!0.5!(C)$);
    \coordinate (PB) at ($(MB) + (1.6,-1)$);
    \coordinate (MC) at ($(A)!0.5!(B)$);
    \coordinate (PC) at ($(MC) + (0,1.6)$);

    \path [name path=LineA_PC] (A) -- (PC);
    \path [name path=LineB_PC] (B) -- (PC);
    \path [name path=LineA_PB] (A) -- (PB);
    \path [name path=LineB_PA] (B) -- (PA);
    \path [name path=LineC_PA] (C) -- (PA);
    \path [name path=LineC_PB] (C) -- (PB);
    \path [name intersections={of=LineA_PC and LineB_PC, by=ABh}];
    \path [name intersections={of=LineA_PB and LineB_PA, by=ABl}];
    \path [name intersections={of=LineB_PA and LineC_PA, by=BCh}];
    \path [name intersections={of=LineB_PC and LineC_PB, by=BCl}];
    \path [name intersections={of=LineC_PB and LineA_PB, by=ACh}];
    \path [name intersections={of=LineC_PA and LineA_PC, by=ACl}];

    \draw [dashed] (B) -- (PA) -- (C);
    \draw [dashed] (A) -- (PB) -- (C);
    \draw [dashed] (A) -- (PC) -- (B);

    \fill (A) circle (3pt);
    \fill (B) circle (3pt);
    \fill (C) circle (3pt);
    \fill (ABh) circle (1.5pt);
    \fill (ABl) circle (1.5pt);
    \fill (BCh) circle (1.5pt);
    \fill (BCl) circle (1.5pt);
    \fill (ACh) circle (1.5pt);
    \fill (ACl) circle (1.5pt);

    \node at (2.2,3) {$P_1$};
    \node at (3.5,-0.5) {$P_2$};
    \node at (0,-0.5) {$P_3$};

    \node at (2.2,1.8) {\footnotesize$1$};
    \node at (1.2,1.8) {\footnotesize$2$};
    \node at (2.7,0.9) {\footnotesize$3$};
    \node at (2.05,0.14) {\footnotesize$4$};
    \node at (0.71,0.9) {\footnotesize$5$};
    \node at (1.35,0.14) {\footnotesize$6$};
\end{tikzpicture}
\end{minipage}\\
\end{tabular}
\vspace{0.6cm}
    \caption{Partial configurations of lines on $S$ in the degenerate types $A_1^2$ (on the left) and $A_1^3$ (on the right).}
    \label{Fig:A1^2A1^3line_arrangement}
\end{figure}

\item \textbf{Type $A_1^4$.} In this case, instead of finding five linearly general points on $S$ arising from the intersection of two lines, we prove directly that $f=\mathrm{id}_S$. According to \cite[Table~3]{sakamaki2010automorphism} (see also \cite[Example~9.2.7]{dolgachev2012classical}), $\mathrm{Aut}(S,D)$ is isomorphic to the permutation group $\mathfrak{S}_4$, which permutes the four $A_1$-singularities. Then, for any $f\in\mathrm{Aut}(S,D)\setminus\{\mathrm{id}_S\}$, there exists at least one of the six quadruple lines $D_i$ such that $f(D_i)\neq D_i$. Since all the lines are fixed by $f$, then $f=\mathrm{id}_S$.

\end{enumerate}

Lastly, we prove that $f=\mathrm{id}_S$ if $(S,\mathbf{b}D)$ is a degeneration of type $N$ or $(A_1^k,N)$, $k=1,2,3$. In these cases, recall that $S=S_0\cup S_1\cup S_2$ with $S_i=V(x_i)$, $i=0,1,2$. First, we observe that each one of the three planes contains a marked line which is not contained in the other two planes. As marked lines are fixed by $f$, we have that $f(S_i)=S_i$ for all $i$. So, it suffices to show that $f|_{S_i}=\mathrm{id}_{S_i}$. Let $\{i,j,k\}=\{1,2,3\}$ and consider the gluing loci $S_i\cap S_j$ and $S_i\cap S_k$. Along each one of these two lines, it is possible to find two distinct points, away from $S_0\cap S_1\cap S_2$, which arise as the intersection of two (or more) marked lines (see \cref{Fig:line_arrangement}). As these four linearly general points are fixed by $f$, we must have that $f|_{S_i}=\mathrm{id}_{S_i}$.
\end{proof}

\begin{theorem}
\label{thm:fiberwise_argument}
The stack $\stackMsp$ is an algebraic space. In particular, the KSBA compactification $\Msp$ is a fine moduli space.
\end{theorem}

\begin{proof}
By \cref{prop:auto}, all the geometric objects parametrized by $\stackMsp$ do not have non-trivial automorphisms. Therefore, by combining \cite[Corollary~8.3.5 and Remark~8.3.6]{olsson2016algebraic}, we have that the stack $\stackMsp$ is an algebraic space. Since $\Msp$ is the coarse moduli space of $\stackMsp$ (see \cite[Definition~11.1.1]{olsson2016algebraic}), we have that $\stackMsp=\Msp$, which implies that $\Msp$ is a fine moduli space.
\end{proof}


\section{Smoothness of the moduli space}
\label{sec:smooth-stack}

\subsection{Deformation theory preliminaries}\label{subsection:preliminary_deformation}

In this section, we prove that the stack $\stackMsp$ is smooth. With the tools developed in \cite{illusie1971complexe,illusie1972complexe,Illusie1972cotangent} by Illusie, Hacking proved that the KSBA compactification of the moduli stack of stable pairs $(\mathbb{P}^2,(\frac{3}{d}+\epsilon) C_d)$, $3\nmid d$, is smooth, where $C_d$ is a smooth curve of degree $d$ in $\mathbb{P}^2$ \cite[Theorem~7.1]{hacking2004compact}. A similar result for the pairs consisting of $\mathbb{P}^1\times\mathbb{P}^1$ with a weighted curve is given in \cite[Theorem~6.6]{deopurkar2021stable}. Our proof follows the strategy in \cite[Section~3]{hacking2004compact}.

\begin{definition}\label{def:ordinarydefobs}
    Let $(S,{\bf b}D)$ be a stable marked cubic surface parametrized by $\Msp$. Let $\mathbb{L}_{S/\mathbb{C}}$ be the cotangent complex of $S\rightarrow\mathrm{Spec}(\mathbb{C})$. The \emph{tangent space/sheaf}, the \emph{deformation space/sheaf}, and the \emph{obstruction space/sheaf} of a projective surface $S$ are defined as
    \begin{align*}
        T_{S}^i&=\mathrm{Ext}^i(\mathbb{L}_{S/\mathbb{C}},\mathcal{O}_{S}),\\
        \T_{S}^i&=\sExt^i(\mathbb{L}_{S/\mathbb{C}},\mathcal{O}_{S}),
    \end{align*}
    for $i=0,1,2$, respectively.
\end{definition}

\begin{remark}
\label{rmk:cotangent_cpx}
In general, the cotangent complex $\mathbb{L}_{X/\mathbb{C}}$ of a variety $X$ is a cochain complex concentrated in non-positive degrees. Referring to \cite[Theorem~10.7]{hacking2001compactification}, we will use the following properties:
\begin{enumerate}

\item $\T_X^0=\sHom(\Omega_{X/\mathbb{C}},\mathcal{O}_X)$, which is the usual tangent sheaf of $X$;

\item The sheaf $\T_X^1$ is supported on the singular locus of $X$;

\item If $X$ is local complete intersection (LCI), then $\T_X^2=0$.

\end{enumerate}
\end{remark}

\begin{remark}\label{Rmk:Q-Gor}
    In the KSBA moduli theory, to preserve the volume of stable pairs in families, one has to work with $\mathbb{Q}$-Gorenstein deformations. Let $p\colon\mathscr{S}\to S$ be the \emph{canonical covering map} from the \emph{canonical covering stack} to $S$ (see \cite[Subection~3.1]{hacking2004compact}). Then, one can define the $\mathbb{Q}$-Gorenstein version of the three spaces/sheaves in \cref{def:ordinarydefobs} by using the cotangent complex $\mathbb{L}_{\mathscr{S}/\mathbb{C}}$ of $\mathscr{S}$, denoted by $T^i_{\mathbb{Q}\mathrm{Gor},S}$ and $\T^i_{\mathbb{Q}\mathrm{Gor},S},i=0,1,2$ (\cite[Notation~3.6]{hacking2004compact}). In our case, for all stable pairs $(S,\mathbf{b}D)$ parametrized by $\Msp$, $S$ is a LCI, and thus Gorenstein. So, their index one covers are the identity. For this reason, our computations will just involve $T_S^i$ and $\T_S^i$.
    For the general treatment of index one covers or $\mathbb{Q}$-Gorenstein deformation theory, we refer to \cite[Section~2.3]{kollar2013singularities}, \cite[Section~3]{hacking2004compact}, and \cite[Section~5]{abramovich2011stable}.
\end{remark}

\subsection{Proof of smoothness}

The next \cref{prop:obstruction_from_div} implies that, to prove smoothness of $\stackMsp$, one only needs to consider the singularity type of $S$, and not the corresponding arrangement of the $27$ marked lines $D$. We start with a technical lemma.

\begin{lemma}\label{lem:3.13_3.14}
    Let $(S,\mathbf{b}D)$ be a stable pair parametrized by $\Msp$. Then, we have that:
    \begin{enumerate}
        \item \label{item:cartier} $D$ is Cartier;
        \item $H^1(S,\mathcal{O}_S(D))=0$.
    \end{enumerate}
\end{lemma}

\begin{proof}
(1) First, we assume that $S$ is smooth or that $(S,\mathbf{b}D)$ is of type $A_1^k$, $k=1,\ldots,4$. We know from \cite[Example~5.2]{gallardo2021geometric} and Part~1 of the proof of \cite[Proposition~5.12]{gallardo2021geometric} that we can find nine planes $H_1,\ldots,H_9\subseteq\mathbb{P}^3$ such that $D=(H_1+\ldots+H_9)|_S$ (these planes are \emph{tritangents}, that is, planes intersecting $S$ in three lines). So, $D$ is Cartier because it is the pullback of a Cartier divisor.

Now, let us assume that $(S,\mathbf{b}D)$ is of type $N$ or $(A_1^k,N)$, $k=1,2,3$. Hence, $S$ is the union of the three coordinate hyperplanes $S_i=V(x_i)$, $i=0,1,2$. Let $U\subseteq S$ be the open subset given by the complement of the double locus. Since $U$ is smooth, the divisor $D$ is Cartier on $U$. So, let $p$ be a point lying on $D$ and $S_0\cap S_1$ (an analogous argument applies for $S_0\cap S_2$ and $S_1\cap S_2$). Let $L_i(1),\ldots,L_i(n_i)\subseteq S_i$ be the lines passing through $p$, and denote by $m_i(1),\ldots,m_i(n_i)$ their multiplicities, $i=0,1$. By checking all the possibilities for $p$ in \cref{Fig:line_arrangement}, one can observe that
\begin{equation}
\label{eq:multiplicities-identity}
\sum_{j=1}^{n_0}m_0(j)=\sum_{j=1}^{n_1}m_1(j).
\end{equation}
For instance, in type $N$, the only possibility is $(m_0(1),m_0(2),m_0(3))=(1,1,1)=(m_1(1),m_1(2),m_1(3))$. In type $(A_1,N)$, we also see the case $(m_0(1),m_0(2),m_0(3))=(1,1,1)$ and $(m_1(1),m_1(2))=(1,2)$, and so on. The equality in \cref{eq:multiplicities-identity} guarantees that, locally at $p$, the divisor $D$ can be cut out by restricting to $S$ planes spanned by pairs of lines, where one line is in $S_0$ and the other one in $S_1$.

(2) Let $H\subseteq\mathbb{P}^3$ be a generic hyperplane. Then, $K_S=-H|_S$, and by Serre's duality we have that
\[
H^1(S,\mathcal{O}_S(D))\cong H^1(S,\mathcal{O}_S(-H|_S-D))^\vee.
\]
So, alternatively, we want to prove the vanishing of $H^1(S,\mathcal{O}_S(-H|_S-D))$. This is a consequence of \cite[Corollary~6.6]{kovacs2010canonical} if we can show that $H|_S+D$ is ample. As $H|_S$ is ample, it suffices to show that $D$ is ample. For $i=0,1,2$, $D|_{S_i}$ is ample because $S_i\cong\mathbb{P}^2$ and $D|_{S_i}$ corresponds to $\mathcal{O}_{\mathbb{P}^2}(9)$ under this isomorphism. In conclusion, $D$ is ample by \cite[Proposition~2.3]{ottem2012ample}.
\end{proof}

\begin{proposition}\label{prop:obstruction_from_div}
    For a pair $(S,\mathbf{b}D)$ parametrized by $\Msp$, let $\mathcal{S}$ be an infinitesimal deformation of $S$ along $\mathbb{C}[\epsilon]/\langle\epsilon^2\rangle\rightarrow\mathbb{C}$. Then, there exists a deformation pair $(\mathcal{S},\mathcal{D})$ of $(S,D)$.
\end{proposition}

\begin{proof}
    By the proof of \cite[Theorem~3.12]{hacking2004compact}, the result is implied by \cref{lem:3.13_3.14}.
\end{proof}

\begin{remark}
By combining \cref{Rmk:Q-Gor} with \cref{prop:obstruction_from_div}, we have that the $\mathbb{Q}$-Gorenstein unobstructedness of the pair $(S,\mathbf{b}D)$ is implied by the unobstructedness of the surface $S$. So, to prove the smoothness of $\stackMsp$ at a point parametrizing the stable pair $(S,\mathbf{b}D)$, it is enough to consider the singularity type of the surface $S$. Therefore, in the proofs that follow, when taking into account the different possibilities for the type of $(S,\mathbf{b}D)$, the discussion for type $(A^k,N)$ will be analogous to the one for type $N$. This is because, in both cases, the surface $S$ is the union of three general planes in $\mathbb{P}^3$. So, we will only consider type $A_1^k$ for $k=1,2,3,4$, type $N$, and $S$ smooth.
\end{remark}

\begin{theorem}\label{thm:smoothstack}
    The algebraic space $\stackMsp$ is smooth. In particular, the projective coarse moduli space $\Msp$ is smooth.
\end{theorem}

\begin{proof}
    To show that $\stackMsp$ is smooth, it suffices to prove that $T^2_{S}=0$, which implies that the obstruction class vanishes. By the Leray spectral sequence $E_2^{p,q}:=H^p(\T^q_{S})\Rightarrow T^{p+q}_{S}$, we have that $T^2_S=0$ if the three cohomology groups
    \[
    H^0(\T^2_{S}),~H^1(\T^1_{S}),~H^2(\T^0_{S})
    \]
    are zero. We prove the vanishing of the above cohomologies in \cref{lem:02_vanishing}, \cref{lem:11_vanishing}, and \cref{lem:20_vanishing}, respectively.
\end{proof}

\begin{lemma}
\label{lem:02_vanishing}
For all pairs $(S,\mathbf{b}D)$ parametrized by $\Msp$, we have that $H^0(\T^2_S)=0$.
\end{lemma}

\begin{proof}
    Since $S$ is a hypersurface in $\mathbb{P}^3$, and hence LCI, by \cref{rmk:cotangent_cpx}~(3) we have that $\T_S^2=0$, which implies that $H^0(\T^2_{S})=0$.
\end{proof}

\begin{lemma}
\label{lem:11_vanishing}
For all pairs $(S,\mathbf{b}D)$ parametrized by $\Msp$, we have that $H^1(\T^1_S)=0$.
\end{lemma}

\begin{proof}
Suppose the pair is of type $A_1^k$ (resp. $S$ is smooth).
By \cref{rmk:cotangent_cpx}~(2), the sheaf $\T_S^1$ is supported on the singular locus of $S$, which is $0$-dimensional (resp. empty), so $H^1(\T_{S}^1)=0$.

For $(S,\mathbf{b}D)$ of type $N$, let $\Delta=\Delta_0\cup\Delta_1\cup\Delta_2\subseteq S$ be the double locus. We have that
\[
\T_S^1\cong\mathcal{O}_\Delta(S)\cong\mathcal{O}_{\mathbb{P}^3}(S)|_\Delta,
\]
where the first and second isomorphisms follow by \cite[Proposition~2.3]{friedman1983global} and \cite[Lemma~1.11]{friedman1983global}, respectively (see \cite[Definition~1.9]{friedman1983global} for the definition of $\mathcal{O}_\Delta(S)$). So, we want to prove that $H^1(\mathcal{O}(S)|_\Delta)=0$.

We have that $\Delta$ is a seminormal variety by \cite[Example~10.12]{kollar2013singularities}. Moreover, $\Delta$ together with the ample line bundle $\mathcal{O}(S)|_\Delta$ can be viewed as a polarized stable toric variety in the sense of \cite[Definition~1.1.5 and Definition~1.1.8]{alexeev2002}. So, the cohomology group $H^1(\mathcal{O}(S)|_\Delta)$ is zero by \cite[Theorem~2.5.1-1]{alexeev2002}.
\end{proof}

\begin{remark}
In the proof of \cref{lem:11_vanishing}, the vanishing of $H^1(\mathcal{O}(S)|_\Delta)$ can also be seen directly using the long exact sequence associated to the short exact sequence in \cite[Chapter~II, Exercise~1.19~(c)]{hartshorne1977algebraic}, where we take $X:=\Delta$, $Z:=\Delta_0$, and $\mathscr{F}:=\mathcal{O}(S)|_\Delta$.
\end{remark}

\begin{lemma}
\label{lem:20_vanishing}
For all pairs $(S,\mathbf{b}D)$ parametrized by $\Msp$, we have that $H^2(\T^0_S)=0$.
\end{lemma}

\begin{proof}
    If the pair $(S,\mathbf{b}D)$ is of type $A_1^k$ or $S$ is smooth, then the normal surface $S$ satisfies $q=p_g=0$ and $h^0(\mathcal{O}_S(-K_S))>0$. Moreover, by \cref{rmk:cotangent_cpx}~(1), $\T^0$ is the tangent sheaf of $S$. Therefore, we can apply \cite[Lemma~11]{Manetti1993} to conclude that $H^2(\T^0_S)=0$.

    For type $N$, write the double locus $\Delta\subseteq S=\cup_{i=0}^2S_i$ as $\cup_{i=0}^2\Delta_i$ with $\Delta_i=\cap_{j\neq i}S_j$ as at the end of \cref{Subsection:stable_pairs_para}. Consider the short exact sequence
    \[
        0\rightarrow\bigoplus_{i=0}^2\mathcal{O}_{S_i}(-\sum_{j\neq i}\Delta_j)\rightarrow\mathcal{O}_S\rightarrow\mathcal{O}_{\Delta}\rightarrow 0.
    \]
    After applying $\sHom_{\mathcal{O}_S}(\Omega_S^1,-)$ to the above and considering the cokernel of the resulting first non-trivial map, we obtain the short exact sequence
    \[
        0\rightarrow\bigoplus_{i=0}^2\T^0_{S_i}(-\sum_{j\neq i}\Delta_j)\xrightarrow{\varphi}\T_S^0\rightarrow\underline{\mathrm{coker}}(\varphi)\rightarrow0.
    \]
    Then, we take the associated long exact sequence in cohomology, obtaining in particular:
    \[
    \ldots\rightarrow\bigoplus_{i=0}^2H^2(\T_{S_i}^0(-\sum_{j\neq i}\Delta_j))\rightarrow H^2(\T_S^0)\rightarrow H^2(\underline{\mathrm{coker}}(\varphi))\rightarrow\ldots
    \]
    We have that $H^2(\underline{\mathrm{coker}}(\varphi))=0$ because $\underline{\mathrm{coker}}(\varphi)$ is supported on $\Delta$, which is one-dimensional. So, it is enough to show that $H^2(\T_{S_i}^0(-\sum_{j\neq i}\Delta_j))$ for $i=0,1,2$. For simplicity, let us assume $i=0$, as the calculation is analogous for $i=1,2$. By Serre's duality, we have that
    \begin{align*}
        H^2(\T^0_{S_0}(-\Delta_1-\Delta_2))&\cong H^0(\mathcal{O}_{S_0}(K_{S_0})\otimes\T_{S_0}^0(-\Delta_1-\Delta_2)^\vee)^\vee\\
        &\cong H^0(\mathcal{O}_{\mathbb{P}^2}(-3)\otimes\Omega_{\mathbb{P}^2}^1(2))^\vee\\
        &\cong H^0(\Omega_{\mathbb{P}^2}^1(-1))^\vee\\
        &\cong0,
    \end{align*}
    where the last isomorphism follows from Bott's formula \cite[Page~4]{okonek1980vector} or, alternatively, from the Euler sequence \cite[Chapter~\rom{2}, Theorem~8.13]{hartshorne1977algebraic}.
\end{proof}

\subsection{Conclusive remarks}

\begin{remark}
In \cite[Theorem~5.17]{gallardo2021geometric}, it was shown that Naruki's compactification $\Mcr$ is isomorphic to the normalization of $\Msp$. As $\Msp$ is smooth by \cref{thm:smoothstack}, the normalization assumption can be removed. Additionally, by \cite[Theorem~1.4]{gallardo2021geometric} we have that $\Mcr$ is isomorphic to the toroidal compactification of an appropriate ball quotient (see \cite{allcock2002complex}). So, it follows that the KSBA compactification $\Msp$ is toroidal without need of normalizing it.
\end{remark}

\begin{remark}\label{Rmk:connectedcomponent}
As another consequence of \cref{thm:smoothstack}, we can give an alternative formulation of the moduli functor $\stackMsp$ in \cref{def:moduli-stack}. Roughly, instead of defining it as the closure of a substack inside a bigger moduli stack of pairs, we can view $\stackMsp$ as a connected component of it. More in detail, let us fix the dimension $d=2$, the weight vector $\mathbf{b}=(\frac{1}{9}+\epsilon,\ldots,\frac{1}{9}+\epsilon)\in\mathbb{Q}^{27}$ with $0<\epsilon\ll 1$, and the volume $(K_S+\mathbf{b}D)^2=243\epsilon^2$, where $S$ is a smooth cubic surface and $D$ is the sum of the $27$ lines. Using the notation in \cite[Section~8]{kollar2023families}, there is a moduli stack $\mathcal{SP}(\mathbf{b},2,243\epsilon^2)$ parametrizing all KSBA stable pairs with these numerical invariants. Let $\mathcal{X}$ be the substack of $\mathcal{SP}(\mathbf{b},2,243\epsilon^2)$ whose corresponding coarse moduli space $X$ is the connected component parametrizing, among others, the stable pairs $(S,\mathbf{b}D)$.

Since the lines do not have embedded deformations in $S$ (this is a consequence of \cite[Chapter~\rom{2}, Exercise~6.9~(a)]{hartshorne2010deformation}), the dimension of the deformation space of the pair $(S,\mathbf{b}D)$ equals the one of $S$. This dimension is given by $h^1(S,\T_S^0)$, which equals $4$ by \cite[Remark~3.13]{huybrechts2023geometry}. So, the irreducible component $X_0$ of $X$ parametrizing pairs $(S,\mathbf{b}D)$ is $4$ dimensional. Moreover, since we know that $X$ is smooth along $X_0$ by the work in \cref{sec:smooth-stack}, we have that $X$ is irreducible. Hence, $X$ coincides with the KSBA compactification $\Msp$ and $\mathcal{X}$ is the claimed alternative formulation of $\stackMsp$.
\end{remark}

\end{document}